\documentclass[preprint,12pt]{amsart}
\usepackage{amssymb}
\usepackage{amsthm}

\def\L {{\rm I\kern-0.16em L}}
\newtheorem{thm}{Theorem}[section]

\newtheorem{lemma}[thm]{Lemma}

\usepackage{amssymb}
\usepackage{amsmath}
\usepackage{graphicx,color}
\usepackage{verbatim}

\usepackage{bbm}
\newcommand{\Besov}[3]{B^{#2}_{#1,#3}}
\newcommand{\normL}{\left| \left|}
\newcommand{\normR}{\right| \right|}
\newcommand{\Ltwo}{{\it{L}_2}}
\newcommand{\Ltwop}{{\it {}^p{L}_{2}}}
\newcommand{\Lp}{{\it{L}_p}}
\newcommand{\Donetwo}{D^{1,2}}
\newcommand{\condEof}[2]{ \E \left( #1 | #2 \right) }
\newcommand{\E}{E}
\newcommand{\F}{\mathcal{F}}
\newcommand{\R}{\mathbbm{R}}
\newcommand{\FF}{\mathbbm{F}}
\newcommand{\Prob}{\mathbbm{P}}
\newcommand{\pDF}{{}^pDF}

\date{\today}

\begin{document}

\title{ Convergence rate for the hedging error of a path-dependent example}
\author{Dario Gasbarra} 
\address{
  University of Helsinki,  Department of Mathematics and Statistics\\
  P.O. Box 68, FI-00014 University of Helsinki \\
  Finland
}
\author{Anni Laitinen} %
\address{
  University of Jyv\"askyl\"a, Department of Mathematics and Statistics\\
  P.O. Box 35, FI-40014 University of Jyv\"askyl\"a\\
  Finland
}
\maketitle

\begin{abstract}
  We consider a Brownian functional \(F=g\bigl(\int_0^T \eta(s) dW_s\bigr)\)
with \(g \in \Ltwo(\gamma)\) and a singular deterministic \(\eta\).
We deduce the $\Ltwo$-convergence rate for the approximation
\(F^{(n)} = E F + \int_0^T \phi^{(n)}(s) dW_s\)
for a class of piecewise constant predictable integrands \(\phi^{(n)}\)
from the fractional smoothness of \(g\) quantified by Besov spaces
and the rate of singularity of \(\eta\).
\end{abstract}

\begin{small}
\end{small}

\section{Introduction}

Let \((W_t)_{t\geq 0}\) be a standard Brownian motion on a complete probability space 
\((\Omega, \F, \Prob)\),
and \( \FF=(\F_t)_{t\geq 0}\) the filtration generated by \(W\) and augmented by the $\Prob$-null sets.
Suppose we have a square-integrable random variable \(F:\Omega \to \R\).
With a suitable choice of a process \((\phi_t)_{t\geq 0}\), we can represent 
\(F\) as
\[
  F= \E F + \int_0^T \phi_t dW_t ,
\]
provided that our \(F\) is \(\F_T\)-measurable.
This would correspond to the perfect hedging of the option \(F\), where \(W\) is the price 
process of the underlying, \(T>0\) the time horizon,  
and \(\phi\) the trading strategy. Using only finitely many trading times
\(0=t_0 < t_1 < \ldots < t_n=T\), denoted by \(\tau_n = (t_i)_{i=0}^n\),
means we get an error
\[
  C(F, \tau_n, \nu) := \int_0^T \phi_t dW_t - \sum_{i=1}^n \nu_{i-1} \left(W_{t_i} - W_{t_{i-1}} \right),
\]
where \(\nu_{i-1}\) is measurable w.r.t. \(\F_{t_{i-1}}\), that is,
\(\nu = (\nu_i)_{i=0}^{n-1} \) is our new hedging strategy, constant between the trading times.

The behaviour of this error as the number of trading times tends to infinity,
depending on the structure of \(F\), has been a topic of intensive research in recent years.

For \(F = f(W_T)\), that is, in the case of European options, where the pay-off depends only on
the stock price at the terminal time, the problem is fairly well understood. The convergence in \(\Ltwo\)
was considered in \cite{ZhangR}, \cite{GobetTemam}, \cite{Gei15}, \cite{GeissGeiss}  and \cite{GeissHujo},
also for non-equidistant time-nets \(\tau\) and more general price processes than \(W\), and 
convergence w.r.t.\ the \(BMO\) norm in \cite{GeissBMO}, employing the \(\Ltwo\) techniques as well.
The convergence with respect to a stronger \(\Lp\) norm for \(p>2\) was investigated in \cite{GeissToivola2}; 
here, missing orthogonality required fundamentally different methods.
For weak convergence, see e.g.\ 
\cite{GobetTemam}, \cite{GeissToivola1}, \cite{HayashiMykland} and \cite{Fukasawa}, to name a few.
The key concept to study the error behaviour is the fractional smoothness of \(f\), which 
links the problem to functional analysis and interpolation theory.
See \cite{GeissGobet}  and the references therein for a more complete overview.

Now, one would like to understand the behaviour of the error
\[
  C(F, \tau_n, \nu) = \int_0^T \phi_t dW_t - \sum_{i=1}^n \nu_{i-1} \left(W_{t_i} - W_{t_{i-1}} \right)
\]
when \(F\) is a square-integrable random variable, not necessarily of the form \(F=f(W_T)\).
A similar question has been considered in \cite{GeissGeissGobet}; there,
the case when the terminal condition \(F\) depends only on finitely many increments of the
price process was investigated in a sightly different setting with backward stochastic differential 
equations.
For a fully path dependent \(F\), the question seems to be open.

This work contributes to this aim by an example of the form  
\begin{equation} \label{eqn:ourF}
  F = g\biggl( \int_0^T \eta(t)dW_t\biggr),
\end{equation}
where \(\eta\) has a singularity \((T-t)^{-\alpha}\) for some \(\alpha \in (0,\frac{1}{2})\) 
at the end.

\section{Preliminaries and notation}

We use the notations \(\Ltwo\) 
for square integrable random variables
and \(\Donetwo \subset \Ltwo\) for random variables in \(\Ltwo\) with square integrable Malliavin derivatives;
for Malliavin calculus and general theory on Brownian motion, 
we refer to \cite{Nualart} and \cite{RevuzYor}. 
The notations \(\Ltwo (\gamma)\) and \(\Donetwo (\gamma)\) are used 
respectively for functions on the real line, with \(\gamma\) denoting the standard Gaussian measure.

For fractional smoothness, we use the index \(\theta \in (0,1)\), and the Besov spaces 
\begin{equation*}
  \Besov{2}{\theta}{q} (\gamma) = \left( \Ltwo (\gamma), \Donetwo (\gamma) \right)_{\theta, q} .
\end{equation*}
These are intermediate spaces between 
\(\Ltwo (\gamma)\) and \(\Donetwo (\gamma)\), obtained by the real interpolation 
(see e.g.\ \cite{B-S}%
), 
with the interpolation parameters \(0 < \theta < 1\) and \(1 \leq q \leq \infty\). 
These spaces have a lexicographical order: %
\[ 
  \Besov{2}{\theta_2}{q_1} (\gamma) \subset \Besov{2}{\theta_1}{q_2} (\gamma) 
\]
for any \(0 < \theta_1 < \theta_2 < 1\) and any \(1 \leq q_1 , q_2 \leq \infty\), and
\[ 
  \Besov{2}{\theta}{q_1} (\gamma) \subset \Besov{2}{\theta}{q_2} (\gamma) 
\]
for any \(0 < \theta < 1\) and any \(1 \leq q_1 \leq q_2 \leq \infty\).
Computing \(\theta\) is elementary for many functions,
see \cite{ToivolaVK} or \cite[Example 2.3]{GeissToivola2} for the standard examples.

\bigskip

Finally, if \(F \in \Donetwo\), 
we use the notation \(\pDF\) for the 
predictable projection of the Malliavin derivative \(DF\), 
and the same notation for the process obtained by 
extending the operator from \(\Donetwo\) to \(\Ltwo\)
(see Section \ref{sec:Appendix}
for details).
Then,
for \(0<t<T\), 
we have
\begin{equation*}
  \E(F | \F_b)- \E(F | \F_a) = \int_a^b (\pDF)_t dW_t
\end{equation*}
for \(0\leq a < b \leq T\).

\section{Results} \label{sec:Results}
For \(T>0\) 
and $\beta\in (0,1)$, we define the function \(\eta \colon [0,T) \to \R\) by setting
\begin{eqnarray*} 
  \eta(t) = (T-t)^{(\beta-1)/2} \; 
  T^{-\beta/2} \sqrt{\beta}.
\end{eqnarray*}
We observe that \(\eta\) has a
singularity at $t=T$, and  $\parallel \eta \parallel_{\Ltwo([0,T])}=1$.

We will consider the random variable 
\[
  F := g\biggl( \int_0^T \eta(t)dW_t\biggr)
\]
with $g\in \Ltwo(\gamma)$. 
Notice that we do not assume continuity of \(g\).

\begin{thm} \label{thm:example}
Let \(0 < \theta < 1\) and \(t_i := \tfrac{i}{n} T\).
If \(g \in \Besov{2}{\theta}{2} (\gamma)\), then
\[
  \normL F - \E F - \sum_{i=1}^n \nu_i \left( W_{t_i} - W_{t_{i-1}} \right) \normR_{\Ltwo} 
  \leq c_2 n^{-\frac{\beta \theta}{2}},
\] 
where
\(\nu_i = (t_i - t_{i-1})^{-1} \int_{t_{i-1}} ^{t_i} \condEof{(\pDF)_s}{\F_{t_{i-1}}} ds\).
\end{thm}

The proof is based on the following observation: 
\begin{lemma} \label{lemma:oneStep}
  For \(0\leq a < b \leq T\), let 
  \[
  \Delta(a,b):= 
  \E(F | {\mathcal F}_b)- \E(F | {\mathcal F}_a) 
  - \frac {W_b-W_a} {b-a}  \int_a^b  \E\bigl( 
  (\pDF)_s \big\vert {\mathcal F}_a \bigr ) ds .
\]
Then
\begin{align*}
& \E\bigl( \Delta(a,b)^2 \bigr) 
\\&=
\E\biggl( \frac 1 {2 (b-a)}  \int_a^b \int_a^b \bigl\{  \E\bigl( 
(\pDF)_t  \big\vert {\mathcal F}_a\bigr) -  \E\bigl(
 (\pDF)_s  \big\vert {\mathcal F}_a\bigr) \bigr\}^2 ds\;dt \biggr) 
\\
& \phantom{=} +  \E\biggl( \int_a^b \bigl\{  (\pDF)_t  - \E\bigl(
(\pDF)_t \big\vert {\mathcal F}_a\bigr)\bigr\}^2 dt \biggr) .
\end{align*}
\end{lemma}

\section{Proofs}
\begin{proof}[Proof of Lemma \ref{lemma:oneStep}]
By a stochastic Fubini theorem,
\begin{align*} 
  \Delta(a,b) 
  &=
  \int_a^b (\pDF)_t dW_t - \frac {W_b-W_a} {b-a} 
  \int_a^b \E\bigl( (\pDF)_s\big\vert{\mathcal F}_a \bigr ) ds  
  \\ &=
  \frac 1 {b-a}
  \int_a^b \biggl(
  \int_a^b \biggl\{  
  (\pDF)_t
  - \E\bigl( (\pDF)_s \big\vert {\mathcal F}_a\bigr) 
  \biggr\}  dW_t \biggr)  ds  
  \\ &= 
  \frac 1 {b-a}
  \int_a^b \biggl( 
  \int_a^b \biggl\{  (\pDF)_t - \E\bigl
  (   (\pDF)_s \big\vert{\mathcal F}_a \bigr) \biggr\} ds \biggr) dW_t  .
\end{align*}

Therefore,
\begin{align*} 
  & \E\bigl( \Delta(a,b)^2 \bigr) 
 \\ &= 
  \frac 1 {(b-a)^2}
  \int_a^b \E \biggl( 
  \int_a^b \biggl\{  (\pDF)_t - \E\bigl
    (   (\pDF)_s \big\vert{\mathcal F}_a \bigr) \biggr\} ds \biggr)^2 dt
  \\ &= 
  \frac 1 {(b-a)^2 }
  \int_a^b \int_a^b \int_a^b
  \E\biggl( 
  \bigl\{ (\pDF)_t
  - \E\bigl(  (\pDF)_z\big\vert {\mathcal F}_a\bigr )
  \bigr\} \times 
  \\ &\phantom{= +}
    \bigl\{
    (\pDF)_t -  \E\bigl( (\pDF)_y \big\vert {\mathcal F}_a\bigr)
    \bigr\} \biggr) dy dz dt
  \\ &= 
  \E\biggl( \frac 1 {2 (b-a)}  \int_a^b \int_a^b \bigl\{  \E\bigl( 
  (\pDF)_t
  \big\vert {\mathcal F}_a\bigr) -  \E\bigl(
  (\pDF)_s  \big\vert {\mathcal F}_a\bigr) \bigr\}^2 ds\;dt \biggr)
  \\ &\phantom{=,}
  + \E\biggl( \int_a^b \bigl\{  (\pDF)_t  - \E\bigl(
  (\pDF)_t
  \big\vert {\mathcal F}_a\bigr)\bigr\}^2 dt \biggr). \; \qedhere
\end{align*}
\end{proof}

\begin{proof}[Proof of Theorem \ref{thm:example}]
Since \(g \in \Ltwo(\gamma)\), we have
$g(x)= \sum\limits_{n=0}^{\infty} c_n \frac{ H_n(x) } {\sqrt{ n!}}$ 
with \(\sum\limits_{n=0}^{\infty}c_n^2 < \infty\).
For \(F\), we then have the chaos expansion
\begin{eqnarray*} 
 F = \sum_{n\ge 0 } I_n(f_n) = \sum_{n } c_n \frac{ 1 } {\sqrt{ n! } } H_n\bigl ( W(\eta) \bigr)
\end{eqnarray*}
with $f_n(t_1,\dots,t_n )= \frac{c_n }{\sqrt{n!}}  \eta(t_1)\eta(t_2) \dots \eta(t_n) $.
With this normalization of the Hermite polynomials, 
and since  $\parallel \eta \parallel_{\Ltwo([0,T])}=1$, we have
$H_n(W(\eta) )= I_n(\eta^{\otimes n } )$. From Lemma \ref{lemma:oneStep} we obtain
\begin{eqnarray*} &&
  \E\bigl( \Delta(a,b)^2 \bigr)
  = \sum_{n\ge 1} \frac {n c_n^2 } { 2 (b-a)} \biggl( \int_0^a \eta(t)^2dt \biggr)^{n-1}
  \int_a^b \int_a^b \bigl( \eta(t)-\eta(s) \bigr)^2 ds dt
  + \\ && 
  \sum_{n\ge 2} c_n^2
  \biggl\{  \biggl( \int_0^b \eta(t)^2 dt \biggr)^{n}
  -  \biggl( \int_0^a \eta(t)^2 dt \biggr)^{n} - n \biggl( \int_0^a \eta(t)^2 dt \biggr)^{n-1}
  \biggl( \int_a^b \eta(t)^2 dt \biggr) \biggr\} 
  \\&&=: A_{a,b} + B_{a,b} .
\end{eqnarray*}

Since
\begin{eqnarray*}
  \int_a^b \eta(t)^2 dt = \frac{\beta}{T^{\beta}} \int_a^b (T-t)^{\beta -1} dt 
  = \left( 1-\frac{a}{T} \right)^{\beta} - \left( 1-\frac{b}{T} \right)^{\beta},
\end{eqnarray*}
we see that
\begin{eqnarray*}
  B_{a,b} 
&=& \sum_{n\ge 2} c_n^2 \left\{
  \left[ 1 - \left(1 - \frac{b}{T} \right)^\beta \right]^{n} -
  \left[ 1 - \left(1 - \frac{a}{T} \right)^\beta \right]^{n} \right. \\ 
&& - n \left.
  \left[ \left(1 - \frac{a}{T} \right)^\beta - \left(1 - \frac{b}{T} \right)^\beta \right]
  \left[ 1 - \left(1 - \frac{a}{T} \right)^\beta \right]^{n-1}
  \right\} \\
&=& \sum_{n\ge 2} c_n^2 
  \int_{1 - \left(1 - \frac{a}{T} \right)^\beta}^{1 - \left(1 - \frac{b}{T} \right)^\beta }
  \int_{1 - \left(1 - \frac{a}{T} \right)^\beta} ^u n (n-1) s^{n-2} ds du .
\end{eqnarray*}
Notice that when \(\beta = 1\), i.e.\ when \(\eta\) is constant, which means that \(F=g(W_T)\), 
the term \(A_{a,b}\) is zero and 
\(
  B_{a,b} = \sum_{n\ge 2} c_n^2 
  \int_{a}^{b}
  \int_{a} ^u n (n-1) s^{n-2} ds du 
\). 
This case was treated in \cite{GeissGeiss} and \cite{GeissHujo}. %

For simplicity, let us suppose that \(T=1\). %
Assuming that 
\begin{equation} \label{eqn:BesovConditionIntegral}
  \sum_{n\ge 2} c_n^2 n (n-1)
  \int_0^1 (1-s)^{1-\theta}  s^{n-2} ds
  = 
  \sum_{n\ge 2} \frac{ n!  c_n^2 }{ (2-\theta)(3-\theta) \dots (n-\theta) } =C < \infty,
\end{equation}
and using \cite[Lemma 3.8]{GeissHujo} 
we achieve, for any time net \(\{t_i\}_{i=1}^m\),
\begin{eqnarray*}
\sum_{i=1}^m \int_{t_{i-1}}^{t_i} \int_{t_{i-1}}^{u} \sum_{n\ge 2} c_n^2 n (n-1) s^{n-2} ds du 
  \leq C_1 \sup_{1 \leq i \leq m} \frac{t_i - t_ {i-1}}{\left(1-t_{i-1}\right)^{1-\theta}}
\end{eqnarray*}
for some \(C_1>0\) depending only on \(C\) and \(\theta\).
Now we fix the time net \(0 = t_0^0 < t_1^0 < \ldots < t_{m-1}^0 < t_m^0 = 1\) with \(t_i^0 := \frac{i}{m}\).
Using \(t_i^{\beta} := 1 - (1 - t_i)^{\beta}\)
we see that
\[
  \sum_{i=1}^m B_{t_{i-1}^0, t_i^0} 
  \leq c_1 \sup_{1 \leq i \leq n} \frac{t_i^{\beta} - t_{i-1}^{\beta}}{\left(1-t_{i-1}^{\beta}\right)^{1-\theta}}
  \leq c_1 m^{-\beta \theta} .
\]

Notice that condition (\ref{eqn:BesovConditionIntegral}) is equivalent to
\(g \in \Besov{2}{\theta}{2}(\gamma)\) (\cite[Theorem 2.2 and proof of Theorem 3.2]{GeissHujo}; 
see also \cite[Theorem 3.1]{GeissToivola2} for another proof),
and that the result so far coincides with the case treated in \cite{GeissHujo}.

For \(A_{a,b}\), we use (\ref{eqn:BesovConditionIntegral}) again in an equivalent form:
\begin{equation} \label{eqn:BesovConditionSum}
  \sum_{n\ge 1} c_n^2 n^{\theta} < C_2 < \infty,
\end{equation}
Then
\begin{align*}
  A_{a,b} &= \sum_{n\ge 1} \frac {n c_n^2 } { 2 (b-a)} \biggl( \int_0^a \eta(t)^2dt \biggr)^{n-1}
  \int_a^b \int_a^b \bigl( \eta(t)-\eta(s) \bigr)^2 ds dt \\
  &= 
  \sum_{n\ge 1} c_n^2 n^{\theta} n^{1-\theta} \left[  1 - \left(1 - a \right)^\beta \right]^{n-1}
  \frac{\beta}{2} \frac{1}{b-a} \times 
  \\ & \phantom{=,}
  \int_{a}^{b} \int_{a}^b\biggl( (1-u)^{(\beta-1)/2} -(1-v)^{ (\beta-1)/2} \biggr)^2 du dv .
\end{align*}
By an elementary computation,
\[
  n^{1-\theta} \left[  1 - \left(1 - a \right)^\beta \right]^{n-1}
  \leq \left( 1 - a \right)^{\beta (\theta -1) }
\]
for all \(n\).

Furthermore,
\begin{align*} 
&  \frac{  \beta \;  } {  2(b-a) } 
  \int_{a}^{b} \int_{a}^b\biggl( (1-u)^{(\beta-1)/2} -(1-v)^{ (\beta-1)/2} \biggr)^2 du dv 
\\ &=
  \beta \int_a^b (1-u)^{\beta-1} du
  -  \frac{  \beta \;  } {  (b-a) } \int_a^b \int_a^b (1-u-v + uv)^{ (\beta-1)/2 } du dv 
\\ &=
  (1-a)^{\beta} - (1-b)^{\beta}
  - \frac {\beta } {(b-a) } \int_a^b \int_{(1-u)(1-a) }^{(1-u)(1-b ) } z^{(\beta-1)/2 } (1-u)^{-1} dv du
\\ &=
  (1-a)^{\beta} - (1-b)^{\beta}
  - \frac {\beta } {(b-a) } \int_a^b  \frac{ 2 }{\beta+1 } (1-u)^{ (\beta-1)/2 }  \times
  \\ & \quad \quad
    \bigl\{ (1-a)^{(\beta+1)/2 }  - (1-b)^{(\beta+1)/2 } \bigr\} du
\\ &=
  (1-a)^{\beta} - (1-b)^{\beta} 
  -\frac {\beta } {(b-a) }   \frac{ 4 }{(\beta +1)^2 }  \bigl\{ 
  (1-a)^{(\beta+1)/2 }  - (1-b)^{(\beta+1)/2 } \bigr\}^2  .
\end{align*}
When $0 \le a < b = 1 $ we obtain
\begin{eqnarray*}
 (1-a)^{\beta}  \biggl(\frac{1-\beta } { 1+ \beta} \biggr)^2 .
\end{eqnarray*}
Otherwise $0 \le a < b < 1$ and by the mid-value theorem
we obtain
\begin{eqnarray*}
 \beta (b-a) ( \xi^{\beta-1} -\eta^{\beta-1} ) \le \beta (\beta-1)  (1-b)^{\beta-2} (b-a)^2
\end{eqnarray*}
for some values $\xi,\eta \in (1-b,1-a)$.

By taking an equally spaced time grid $(t_k = k /N : \;k=0,\dots,N)$, 
we have that the sum of the mean square errors over the interval $[0,1]$
 is bounded by
\begin{eqnarray*} &&
  N^{-\beta}   \biggl(\frac{1-\beta } { 1+ \beta} \biggr)^2    + \beta (\beta-1) N^{-2} \sum_{k=1}^{N-1} 
  \biggl( \frac k N \biggr)^{\beta-2} \\ &&
  \le  N^{-\beta} \biggl\{   \biggl(\frac{1-\beta } { 1+ \beta} \biggr)^2 + \beta (\beta-1) \zeta(2-\beta) \biggr\} ,
\end{eqnarray*}
where $\zeta(t)= \sum\limits_{k=1}^{\infty} k^{-t}\;$ is Riemann's zeta function.

Thus for $\beta < 1$,
\[
  \sum_{i=1}^m A_{t_{i-1},t_i} \leq C_2 m^{-\beta(\theta -1)} m^{-\beta}
\biggl\{   \biggl(\frac{1-\beta } { 1+ \beta} \biggr)^2 + \beta (\beta-1) \zeta(2-\beta) \biggr\}  = C_3 m^{-\beta \theta} . \qedhere
\]
\end{proof}

\section{Appendix} \label{sec:Appendix}

Here  $\FF=\FF^W$ is the augmented
Brownian filtration and ${\mathcal P}$ is the $\FF$-predictable $\sigma$-algebra
on \((\Omega \times [0,T])\)
generated by the left continuous $\FF$-adapted processes
(for the notation and concepts in this section, see \cite{RevuzYor}).
Note that in the Brownian filtration, the $\FF$-predictable and $\FF$-optional $\sigma$-algebrae coincide.
This is not so in general for the filtration generated by a L\'evy process.

Let 
\begin{align*}
 \Ltwop(\Omega \times [0,T]):=\Ltwo( \Omega \times [0,T], {\mathcal P}, dP \times dt)
\end{align*}
the subspace of  $\FF$-predictable integrands $Y_t(\omega)$ with
\begin{align*}
  E\biggl(   \int_0^T Y_t^2 dt \biggr) < \infty.
\end{align*}
Note that $\Ltwop( \Omega \times [0,T]) $ is  a closed subspace of 
$\Ltwo(\Omega \times [0,T], {\mathcal F}_T^W  \otimes {\mathcal B}([0,T]), dP \times dt)$.

For any process 
$X_t(\omega) \in \Ltwo(\Omega\times [0,T], {\mathcal F}^W_T \otimes {\mathcal B}([0,T]), dP \times dt)$,
there exists  the  predictable projection 
${}^p X  \in  \Ltwop(\Omega\times [0,T])$
such that for any $\FF$-predictable stopping time $\tau(\omega)$ 
(notice that in the Brownian filtration all stopping times are predictable),
\begin{align*}
  E ( X_{\tau} | {\mathcal F}_{\tau-} )(\omega){\bf 1}( \tau(\omega) <  \infty) 
  ={}^p X_{\tau(\omega)}(\omega) {\bf 1}( \tau(\omega) <  \infty),
\end{align*}
where
\begin{align*}
  {\mathcal F}_{\tau-}= \sigma\{  A \cap \{t<\tau \}:   t \ge 0,   A \in {\mathcal F}_t  \}.
\end{align*}
By the It\^o representation theorem,
if $F(\omega)\in \Ltwo(\Omega, {\mathcal F}^W_T, P)$, then there is an unique $Y_t(\omega) \in \Ltwop( \Omega \times [0,T])$
such that
\begin{align*}
  F= E(F) + \int_0^T Y_t dW(t) ,
\end{align*}
where the stochastic integral is an It\^o integral, and by the It\^o isometry 
\begin{align*}
  E(  F^2) = E(F)^2 + E\biggl(  \int_0^T Y_t^2 dt \biggr) .
\end{align*}

We show that %
$Y_t = ( {}^p D F)_t$,
where 
\[
  {}^p D :\Ltwo( \Omega) \to   \Ltwop( \Omega \times [0,T])%
\]
is the closure of the  operator defined for $F\in D^{1,2}$ such that $F  \mapsto  {}^p ( DF )$,
and $\mbox{Dom}( {}^p D ) =\Ltwo( \Omega,{\mathcal F}^W_T, P)$.

\begin{lemma} For $F\in D^{1,2}$, define  ${}^p D F = {}^p( DF)$.  
In other words, for any $\F$-predictable stopping time \(\tau\),
\begin{align*}
   {}^p D_{\tau} F(\omega){\bf 1}( \tau(\omega) <  \infty) 
   = E( D_{\tau} F | {\mathcal F}_{\tau-} ) (\omega){\bf 1}( \tau(\omega) <  \infty) .
\end{align*}

In particular, if $\tau(\omega) \equiv t$ is deterministic,
\begin{align*}
   {}^p D_{t} F(\omega) = E( D_{t} F | {\mathcal F}_{t-} ) (\omega)=E( D_{t} F | {\mathcal F}_{t} ) (\omega) .
\end{align*}

The operator $\;{}^p D$ is closable.
\end{lemma}

\begin{proof}
Let $F_n \in D^{1,2}$ with $F_n \stackrel{ \Ltwo(\Omega) }{\longrightarrow} 0$
and ${}^p D F_{n}  \stackrel{ \Ltwo(\Omega\times [0,T]) }{\longrightarrow} Y $.
Then necessarily $Y=0$.%

This is immediate from the It\^o isometry and It\^o-Clarck-Ocone formula 
for $F_n\in D^{1,2}$:
\begin{align*}
  F_n- E(F_n) = \int_0^T   {}^p ( D F_n )_t  dW_t 
  = \int_0^T   E(  D_t F_n | {\mathcal F}_t) dW_t  ,
\end{align*}
and by the It\^o isometry, 
\begin{align*}
  E\bigl( \bigl \{F_n- E(F_n) \bigr\}^2 \bigr) 
  &= 
  E\biggl( \biggl\{\int_0^T   E(  D_t F_n | {\mathcal F}_t) dW_t\biggr\}^2 \biggr) \\
  &= 
  \int_0^T  E\bigl( \bigl\{ {}^p ( D F_n )_t \bigr\}^2 \bigr) dt 
\end{align*}
    so that 
    \({}^p ( D F_n ) \to 0\) in $\Ltwo(\Omega \times  [0,T])$ as $F_n \to 0$ in $\Ltwo(\Omega)$.
\end{proof}

Let now $F\in \Ltwo(\Omega,{\mathcal F}^W_T,P)$ with 
It\^o representation
\begin{align*}
 F= E( F ) + \int_0^T Y_t dW_t ,
\end{align*}
where $Y \in \Ltwop( \Omega \times [0,T])$, and let $F_n \in D^{1,2}$ with $F_n \stackrel{\Ltwo(\Omega) }{\to} F$.
Then by It\^o isometry and It\^o-Clarck-Ocone formula necessarily 
\[
  {}^p D F_n :={}^p (DF_n)   \stackrel{\Ltwo(\Omega\times [0,T] ) }{\longrightarrow} Y,
\]
which means that the extension ${}^p D F = Y$ is well defined.

\bigskip

This gives a practical procedure to compute the martingale representation of $F\in \Ltwo(\Omega)$ : write formally $DF$ 
using the chaos expansion or the differentiation rules of Malliavin calculus. In general $DF$ does not need to
be a random function, it may live in a space of random distributions. Nevertheless when  we 
compute
the predictable projection of the formal chaos expansion of $DF$, 
by evaluating $E( D_t F |{\mathcal F}_t)$, we always obtain
a nice integrand in $\Ltwop( \Omega \times [0,T])$.

\end{document}